\documentclass[a4paper,english]{lipics-v2021}
\hideLIPIcs
\nolinenumbers

\usepackage{amsthm}
\usepackage{amsmath}
\usepackage{amssymb}
\usepackage{amsfonts}
\usepackage{epsfig}
\usepackage{graphics}
\usepackage{graphicx}
\usepackage{booktabs}
\usepackage{pdfpages}
\usepackage{tikz}
\usepackage{pgfplots}
\usepackage[T1]{fontenc}

\nolinenumbers

\usepackage{complexity}
\usepackage{url}

\usepackage{latexsym}
\usepackage{psfrag}
\usepackage{enumerate}
\usepackage{here}
\usepackage{esvect}
\usepackage[update,prepend]{epstopdf}
\usepackage{anyfontsize}
\usepackage{url}
\usepackage{hyperref}

\usepackage{microtype}

\newcommand{\etal}{{et~al.}}
\newcommand{\ie}{{i.e.}}
\newcommand{\eg}{{e.g.}}

\newcommand{\area}{{\rm Area}}
\newcommand{\conv}{{\rm conv}}

\newcommand{\dist}{{\rm dist}}

\newcommand{\NN}{\mathbb{N}} 
\newcommand{\ZZ}{\mathbb{Z}} 
\newcommand{\RR}{\mathbb{R}} 
\newcommand{\eps}{\varepsilon}

\newtheorem{problem}[theorem]{Problem}

\def\C{\mathcal C}
\def\G{\mathcal G}

\def\R{\mathcal R}
\def\S{\mathcal S}

\usepackage[noline,linesnumbered]{algorithm2e}
\SetArgSty{textrm}
\let\oldnl\nl
\newcommand{\nonl}{\renewcommand{\nl}{\let\nl\oldnl}}
\makeatletter
\def\TitleOfAlgo{\@ifnextchar({\@TitleOfAlgoAndComment}{\@TitleOfAlgoNoComment}}
\def\@TitleOfAlgoAndComment(#1)#2{\nonl\hspace*{-1.5em}#2 #1\;}
\def\@TitleOfAlgoNoComment#1{\nonl\hspace*{-1.5em}#1\;}
\makeatother
\DontPrintSemicolon

\title{Partitioning Complete Geometric Graphs on Dense Point Sets
  into Plane Subgraphs}

\titlerunning{Partitioning complete geometric graphs into plane subgraphs}

\author{Adrian Dumitrescu}
{Algoresearch L.L.C., Milwaukee, WI, USA}
{ad.dumitrescu@algoresearch.org}
{0000-0002-1118-0321}
{}
\author{J\'anos Pach}
       {Alfr\'ed R\'enyi Institute of Mathematics, Budapest, Hungary
         and EPFL, Lausanne, Switzerland}
{pach@renyi.hu}
{0000-0002-2389-2035}
{}

\authorrunning{Adrian Dumitrescu and J\'anos Pach}

\Copyright{Adrian Dumitrescu and J\'anos Pach}

\funding{Work on this paper was supported by ERC Advanced Grant no. 882971  ``GeoScape."
  Work by the second named author was partially supported by NKFIH
  (National Research, Development and Innovation Office) grant K-131529.}



\begin{document}

\maketitle

\begin{abstract}
A \emph{complete geometric graph} consists of a set $P$ of $n$ points
in the plane, in general position, and all segments (edges) connecting
them. It is a well known question of Bose, Hurtado, Rivera-Campo, and Wood,
whether there exists a positive constant $c<1$, such that every
complete geometric graph on $n$ points can be partitioned into at most
$cn$ plane graphs (that is, noncrossing subgraphs). We answer this
question in the affirmative in the special case where the underlying
point set $P$ is \emph{dense}, which means that the ratio between the
maximum and the minimum distances in $P$ is of the order of $\Theta(\sqrt{n})$.  
\end{abstract}

\section{Introduction} \label{sec:intro}

A set of points in the plane is said to be:
(i)~in \emph{general position} if no $3$ points are collinear; and
(ii)~in \emph{convex position} if none of the points lies in the convex hull of the other points.
For a set $A$ of $n$ points in the plane, consider the ratio
\[ D(A)= \frac{\max\{\dist(a,b) : a,b \in A, a \neq b\}}{\min\{\dist(a,b) : a,b \in A, a \neq b\}}, \]
where $\dist(a,b)$ is the Euclidean distance between points $a$ and $b$.
We assume throughout this paper and without loss of generality that the minimum pairwise distance is $1$.
In this case $D(A)$ is the diameter of $A$.
A standard volume argument shows that if $A$ has $n$ points, then $D(A) \geq \alpha_0 \, n^{1/2}$,
with
\begin{equation}\label{alfa0}
\alpha_0:= 2^{1/2} 3^{1/4} \pi^{-1/2} \approx 1.05,
\end{equation}
provided that $n$ is large enough; see~\cite[Prop.~4.10]{Va92}.
On the other hand, a $\sqrt{n} \times \sqrt{n}$ section of the integer lattice shows that
this bound is tight up to a constant factor.

Given $n$ points in the plane, in general position, the graph obtained by connecting certain point-pairs
by straight-line segments is called a \emph{geometric graph} $G$.
If no two segments (edges) of $G$ cross each other, then $G$ is a \emph{plane graph}.
A graph of the form $K_{1s}$, where $s \geq 0$, is a special plane graph, called a \emph{star};
in particular, a single vertex is a star with no leaves.
A graph in which every connected component is a star is called a \emph{star-forest};
see, \eg, \cite{AK82}.

Obviously, every complete geometric graph of $n$ vertices can be decomposed into $n-1$ plane stars.
In the present note, we address the following problem of Bose, Hurtado, Rivera-Campo, and Wood~\cite{BHRW06},
raised almost 20 years ago.

\begin{problem}\rm{\cite{BHRW06}} \label{prob:partition}
  Does there exist a positive constant $c<1$ such that every complete geometric graph on $n$ vertices
  can be partitioned into at most $cn$ plane subgraphs?
\end{problem}

An $n$-element point set $A$ satisfying the condition $D(A) \leq \alpha \, n^{1/2}$,
for some constant $\alpha\ge\alpha_0$, is said to be $\alpha$-\emph{dense}; see the works of
Edelsbrunner, Valtr, and Welzl~\cite{EVW97}, Kov\'{a}cs and T{\'{o}}th~\cite{KT20}, and Valtr~\cite{Va94}.
(Note, the larger $\alpha$ becomes, the ``less dense'' the set gets.)

Here, we solve Problem~\ref{prob:partition} for dense point sets.

\begin{theorem} \label{thm:partition}
  Let $A$ be an $\alpha$-dense point set of $n$ points in general position in the plane,
  and let $K_n=K_n[A]$ denote the complete geometric graph induced by $A$.
  Then (the edge set of) $K_n$ can be decomposed into at most $cn$ plane subgraphs,
  where $c=c(\alpha)<1$ is a constant. Specifically, we have
\begin{equation} \label{eq:c(alpha)}
  c(\alpha) \leq  1 - \Omega\left(\alpha^{-12}\right).
\end{equation}
Each of these plane graphs is either a star or a plane union of two stars.
\end{theorem}

Let $A$ be a randomly and uniformly distributed set of $n$ points in the unit square. With probability tending to $1$,
as $n\rightarrow\infty$, the order of magnitude of the minimum distance in $A$ will be much smaller than $n^{-1/2}$.
Therefore, $D(A)$ will be larger than $\alpha \, n^{1/2}$, for every $\alpha$, provided that $n$ is sufficiently large,
and Theorem~\ref{thm:partition} cannot be applied to $K_n[A]$. Nevertheless, $A$ almost surely contains
a linear-size $\alpha'$-dense subset, for a suitable constant $\alpha'$, and we can easily deduce the following statement,
which is also implied by a result of Valtr~\cite[Thm.~14]{Va96} in conjunction with Lemma~\ref{lem:cf-partition} below.
In Section~\ref{sec:remarks} we provide an alternative proof of Corollary~\ref{cor:random}.

\begin{corollary} \label{cor:random}
 Let $A$ be a set of $n$ random points uniformly distributed in $[0,1]^2$, and let $n\rightarrow\infty$.
 There exists an absolute constant $c<1$ such that, with probability tending to $1$,
 the complete geometric graph induced by $A$ can be decomposed into at most $cn$ plane subgraphs.
\end{corollary}

There is an intimate relationship between the above problem and another old question in combinatorial geometry,
due to Aronov, Erd\H{o}s, Goddard, Kleitman, Klugerman, Pach, and Schulman~\cite{AEG+91}.
Two segments are said to \emph{cross} each other if they do not share an endpoint and they have
an interior point in common.

\begin{problem}\rm{\cite{AEG+91}} \label{prob:crossing}
  Does there exist a positive constant $c<1/2$ such that every complete geometric graph on $n$ vertices
  has $cn$ pairwise crossing edges?
\end{problem}

In the general case, Pach, Rubin, and Tardos~\cite{PRT19} established the existence of at least
$n/2^{O(\sqrt{\log{n}})}=n^{1-o(1)}$ pairwise crossing edges.
For dense point sets, a better, but still sublinear, lower bound was established by~Valtr\cite{Va96}.
From the other direction, Aichholzer,  Kyn{\v{c}}l, Scheucher, Vogtenhuber, and Valtr~\cite{AKS+22}
constructed $n$-element point sets that do not contain more than $8 \lceil \frac{n}{41} \rceil$
pairwise crossing edges; see also~\cite{ES19}.

Problems~\ref{prob:partition} and~\ref{prob:crossing} are connected by the following simple,
but important finding of Bose~\etal~\cite{BHRW06}.

\begin{lemma} \label{lem:cf-partition} {\rm \cite{BHRW06}}
  If a complete geometric graph $K_n$ of $n$ vertices has $p$ pairwise crossing edges,
  then $K_n$ can be partitioned into $n-p$ plane trees and, hence, into $n-p$ plane subgraphs.
\end{lemma}

In view of this statement, a positive answer to Problem~\ref{prob:crossing} would immediately imply
our Theorem~\ref{thm:partition}. Lacking such an answer, we need to take a different approach,
which is described in the next section.
\smallskip

All point sets appearing in this note are in general position, and the logarithms are in base~$2$.
For any triple of points $a,b,c$, let $\Delta{abc}$ denote the triangle with vertices $a,b,c$.

\section{Proof of Theorem~\ref{thm:partition}} \label{sec:proofs}

In this section we prove Theorem~\ref{thm:partition}.
We start with a basic observation.

\begin{lemma} \label{lem:combined}
  Let $A$ be a set of $n$ points in general position in the plane, and let $B\subseteq A$ where $|B|=b$.
  Suppose that the complete geometric graph $K_{b}[B]$ induced by $B$ can be decomposed into $b-p$
  plane subgraphs, for some $p \geq 1$.
  Then $K_n[A]$, the complete geometric graph induced by $A$, can be decomposed into $n-p$ plane subgraphs.
\end{lemma}
\begin{proof}
  Consider the $n-b$ stars centered at points in $A \setminus B$ together with the $b-p$ plane
  subgraphs in the  decomposition of  $K_{b}[B]$, and delete duplicate edges. 
\end{proof}

In view of Lemma~\ref{lem:combined}, to establish Theorem~\ref{thm:partition}, it is enough to find 
a large subset $B\subseteq A$ that can be decomposed into relatively few plane graphs.
Instead of Lemma~\ref{lem:cf-partition}, we use the following result, whose proof is included for completeness.

\begin{lemma} \label{lem:special} {\rm (Pach, Saghafian, and Schnider~\cite{PSS23})}.
  Let $B = \bigcup_{i=1}^4 B_i$ be a set of $4m$ points in general position in the plane,
  where $|B_1|=|B_2|=|B_3|=|B_4| = m$, such that for every choice $p_i \in B_i$, for $i=1,2,3,4$,
  we have that $p_4$ lies inside the convex hull of $\{p_1,p_2,p_3\}$. 
  Then the complete geometric graph $K_{4m}[B]$ induced by $B$ can be decomposed into at most $3m$ plane subgraphs.
\end{lemma}
\begin{proof}
  We decompose the complete geometric graph $K_{4m}[B]$ into $3m$ plane star-forests,
  which come in three families; see Fig.~\ref{fig:3m}:

\begin{enumerate} \itemsep 1pt
\item all stars emanating from points in $B_1$ connecting to all points in $B_1$ and $B_2$
  together with all stars emanating from points in $B_3$ connecting to all points in $B_3$ and $B_4$
\item all stars emanating from points in $B_2$ connecting to all points in $B_2$ and $B_3$
  together with all stars emanating from points in $B_4$ connecting to all points in $B_4$ and $B_1$
\item all stars emanating from points in $B_1$ connecting to all points in $B_1$ and $B_3$
  together with all stars emanating from points in $B_2$ connecting to all points in $B_2$ and $B_4$
\end{enumerate}

\begin{figure}[htbp]
\centering
\includegraphics[scale=0.75]{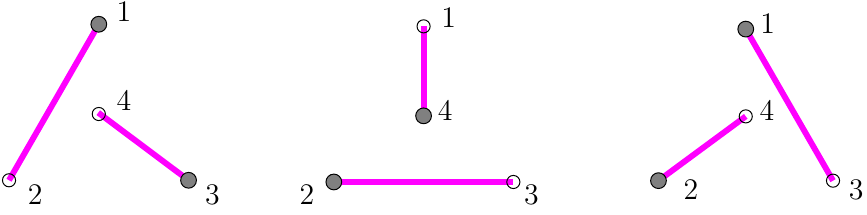}
\caption {Sketch of the $3m$ plane subgraphs in Lemma~\ref{lem:special}.}
\label{fig:3m}
\end{figure}

Observe that these stars cover all edges of $K_{4m}[B]$.
The first family is the union of $m$ plane subgraphs: indeed,
no star connecting a point in $B_1$  to every point in $B_1$ and $B_2$
crosses any star connecting a point in $B_3$ to every point in $B_3$ and $B_4$ by the assumption.
Therefore, these stars can be grouped in pairs such that each pair forms a plane \emph{star forest}.

Similarly, the second and third families also consist of $m$ plane subgraphs, each.
Removing duplicate edges, we obtain a decomposition of $K_{4m}[B]$ into $3m$ plane star forests.
\end{proof}

We show that every $\alpha$-dense $n$-element point set $A$ contains a subset $B$ satisfying the 
conditions in Lemma~\ref{lem:special} with $m=\Omega(n)$.  
The overall idea is to find four large (linear-size) subsets of $A$ clustered around four points \emph{not}
in convex position, as depicted in Fig.~\ref{fig:grid}\,(right). Once this favorable configuration is found,
it yields a partition of the corresponding complete geometric graph into a small number of plane subgraphs.
This partition is extended to a partition of the complete geometric graph of the original set into a small number
of plane subgraphs. We next provide the details.

Let $k=k(\alpha) \geq 3 \alpha^2$ be an increasing function of $\alpha$ to be specified later, and set
\begin{equation} \label{eq:parameters}
  n_0 = \lceil 12 k^2/\alpha^2 \rceil.
\end{equation}
We may assume without loss of generality that $k(\alpha)$ takes integer values
(by applying the ceiling function if needed). 
We distinguish between two cases: $n \leq n_0$, and $n \geq n_0$.
Suppose first, that $n \leq n_0$.
Recall that there is a decomposition of $K_n[A]$ into $n-1$ plane subgraphs that are stars.
Note that $n-1 \leq c n$ for $n \leq n_0$ provided that $c<1$ is large enough: indeed,
$n(1-c) \leq n_0(1-c) \leq 1$ for $c \geq 1 -1/n_0$.

Suppose next, that $n \geq n_0$.
Let $A$ be an $n$-element $\alpha$-dense set.
Since $D(A) \leq \alpha \sqrt{n}$, we may assume that $A$ is contained in an axis-aligned square $Q$
of side-length $\alpha \sqrt{n}$.
Subdivide $Q$ into $k^2$ axis-parallel squares, called \emph{cells},
of side-length $\alpha \sqrt{n}/k$. Let $\Sigma$ be the set of all $k^2$ cells in $Q$.
We may assume without loss of generality that no point in $A$ lies on a cell boundary.

\begin{lemma} \label{lem:cell-ub}
Each cell in $\Sigma$ contains at most $\frac{2 \alpha^2}{k^2} \, n$ points of $A$.
\end{lemma}
\begin{proof}
  Let $\sigma \subset Q$ be any cell and $\sigma'$ be the axis-aligned square concentric with $\sigma$
  and whose side-length is $\frac{\alpha \sqrt{n}}{k} + 1$.
  Obviously,  $\sigma'$ contains all disks of radius $1/2$ centered at the points of $A \cap \sigma$.
  Moreover, since $A$ is $\alpha$-dense, these $n$ disks are interior disjoint.
Moreover, $\sigma'$ is a so-called \emph{tiling domain}, \ie, a domain that can be used to
tile the whole plane. Let $m$ denote the number of points of $A \cap \sigma$.
A packing of $m$ congruent disks of radius $1/2$ in $\sigma'$ requires~\cite[Ch.~3.4]{FFK23} 
that $m \frac{\pi}{4} \leq \frac{\pi}{\sqrt{12}} \, \area(\sigma')$,
which yields (by using that $n \geq n_0$):
\[ m \leq \frac{2}{\sqrt3} \left( \frac{\alpha \sqrt{n}}{k} + 1 \right)^2 \leq
\frac{2}{\sqrt3} \left( 1 + \frac{1}{\sqrt{12}} \right)^2 \frac{\alpha^2}{k^2} \, n
\leq \frac{2 \alpha^2}{k^2} \, n. \qedhere \]
\end{proof}

\smallskip
A cell $\sigma \in \Sigma$ is said to be \emph{rich} if it contains at least $n/(3k^2)$ points of $A$,
and \emph{poor} otherwise. Let $\R \subset \Sigma$ denote the set of rich cells.

\begin{lemma} \label{lem:rich-lb}
  There are at least $\frac{k^2}{3\alpha^2}$ rich cells; that is, $|\R| \geq \frac{k^2}{3\alpha^2}$.
\end{lemma}
\begin{proof}
Let $r=|\R|$ denote the number of rich cells. Assume for contradiction that $r < \frac{k^2}{3\alpha^2}$.
By Lemma~\ref{lem:cell-ub} the total number of points of rich cells is at most
\[ r \cdot \frac{2 \alpha^2}{k^2} \, n < \frac{k^2}{3\alpha^2} \cdot \frac{2 \alpha^2}{k^2} \, n = \frac23 \, n. \]
The total number of points of poor cells is less than
\[ k^2 \cdot \frac{n}{3k^2} = \frac{n}{3}. \]
Thus the total number of points of $A$ is strictly less than $n$, a contradiction that \linebreak
completes the proof.
\end{proof}

\begin{lemma} \label{lem:4rich}
There exist four rich cells $\sigma'_1,\sigma'_2,\sigma'_3,\sigma'_4$, such that:
  \begin{itemize}
  \item for any four points $a_i \in \sigma'_i \cap A$, $i=1,2,3,4$,
    we have $a_4 \in \Delta{a_1 a_2 a_3}$.
    \end{itemize}
\end{lemma}

Before presenting our proof of Lemma~\ref{lem:4rich}, we sketch a simple alternative proof using
a very powerful tool: the Furstenberg-Katznelson theorem, also called \emph{Density Hales-Jewett theorem}.
However, it is not strong enough to yield the quantitative statement in Theorem~\ref{thm:partition}.
In a sufficiently large dense subset of a grid in $\ZZ^d$, for any fixed $d$ and $s$, one can always find
a $s \times s$ grid as a subgrid. 
The case $d=1$ corresponds to a classical result of Szemer\'{e}di~\cite{Sz75}. 
A higher dimensional generalization of Szemer\'{e}di's density theorem was obtained by
Furstenberg and Katznelson \cite{FK78}; see also \cite{N95}.
Their proof uses infinitary methods in ergodic theory.
A more recent combinatorial proof of this statement can be found in~\cite{Polymath12},
but the resulting bound is huge (a tower of $2$'s of polynomial height).

\begin{theorem} {\rm (Furstenberg--Katznelson \cite{FK78}).} \label{thm:FK} 
For all positive integers $d$, $s$ and  every $c>0$, there exists
a positive integer $N=N(d,s,c)$  with the following property: 
every subset $X$ of $\{1,2,\ldots,N\}^d$ of size at least
$c N^d$ contains a homothetic copy of $\{1,2,\ldots,s\}^d$.
\end{theorem} 
\begin{figure}[htbp]
\centering
\includegraphics[scale=0.7]{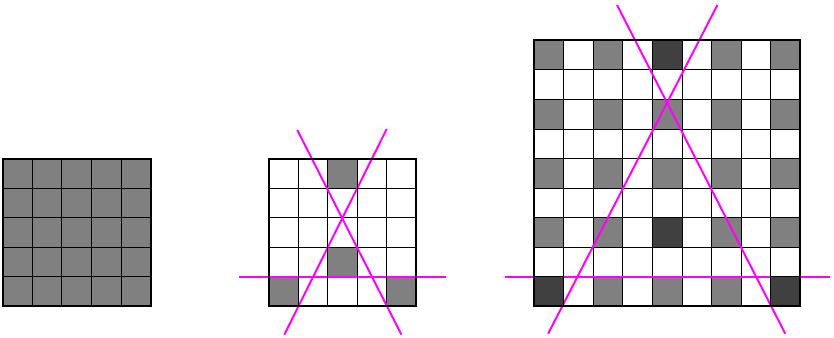}
\caption {Left: A $5 \times 5$ subgrid of rich cells (shaded) and a relevant subset of four cells.
  \linebreak
Right: A~$5 \times 5$ subgrid of rich cells with some separation.}
\label{fig:5x5}
\end{figure}

To deduce Lemma~\ref{lem:4rich}, apply Theorem~\ref{thm:FK} with $d=2$, $s=5$, and $c=1/(3 \alpha^2)$,
to the set $\Sigma$ of cells in $Q$ and its subset $\R$ of rich cells.
That is, let $k:= N(2,5,1/(3 \alpha^2))$. By~\eqref{eq:parameters},
if $n$ is large enough, this setting ensures the existence of a $5 \times 5$ subgrid of rich cells.
Fig.~\ref{fig:5x5} shows the four selected rich cells satisfying the requirements in Lemma~\ref{lem:4rich}.
Note that a separation between subgrid cells, if any, does not interfere with the result.

\subparagraph{Proof of Lemma~\ref{lem:4rich}.}
  Let $P=\conv(\R)$. Note that $P$ is a lattice polygon whose vertices are in the
  $(k+1) \times (k+1)$ grid $\G$ subdividing $Q$.
  Let $\C \subset \R$ denote the set of rich cells incident to vertices of $P$:
  we have |$\C| \leq v(P)$.  By a well-known result,
  $P$ has
\begin{equation} \label{eq:grid-gon}
  v(P) \leq c' k^{2/3}
\end{equation}
  vertices in $\G$, where $c'>0$ is a suitable constant;
  see, \eg, \cite[Exercise 2, p.~34]{Mat02}.
  (A more precise estimate on the number of vertices was given by
  Acketa and \v{Z}uni{\'c}~\cite{AZ95}:
  \begin{equation}  \label{eq:v(P)}
    v(P) \leq 12 (4 \pi^2)^{-1/3} k^{2/3} +O(k ^{1/3} \log{k}).
  \end{equation}
  However, here we need a non-asymptotic upper bound.)

Choose an arbitrary element of $\C$, say a leftmost one, and denote it by $\sigma_0$.
Label the remaining elements of $\C$ in clockwise order around the boundary of $P$
as $\sigma_1,\sigma_2,\ldots,\sigma_{|\C|-1}$.
  Consider the convex sets $\tau_1,\tau_2,\ldots,\tau_{2|\C|-3}$ defined as follows:
\begin{align} \label{eq:tau}
  \tau_j &= \conv(\sigma_0 \cup \sigma_j), \ \ j=1,2,\ldots,|\C|-1, \\
  \tau_{|\C|+j-1} &= \conv(\sigma_j \cup \sigma_{j+1}), \ \ j=1,\ldots,|\C|-2, \text{ and let} \\
    K &= \bigcup_{j=1}^{2|\C|-3} \tau_j.
\end{align}

We refer to $K$ as the \emph{star triangulation} from the boundary cell $\sigma_0$.
Let $\S$ denote the set of segments that appear on the boundaries of $\tau_1,\tau_2,\ldots,\tau_{2|\C|-3}$.
See Fig.~\ref{fig:grid} for an example.

\begin{figure}[htbp]
\centering
\includegraphics[scale=0.31]{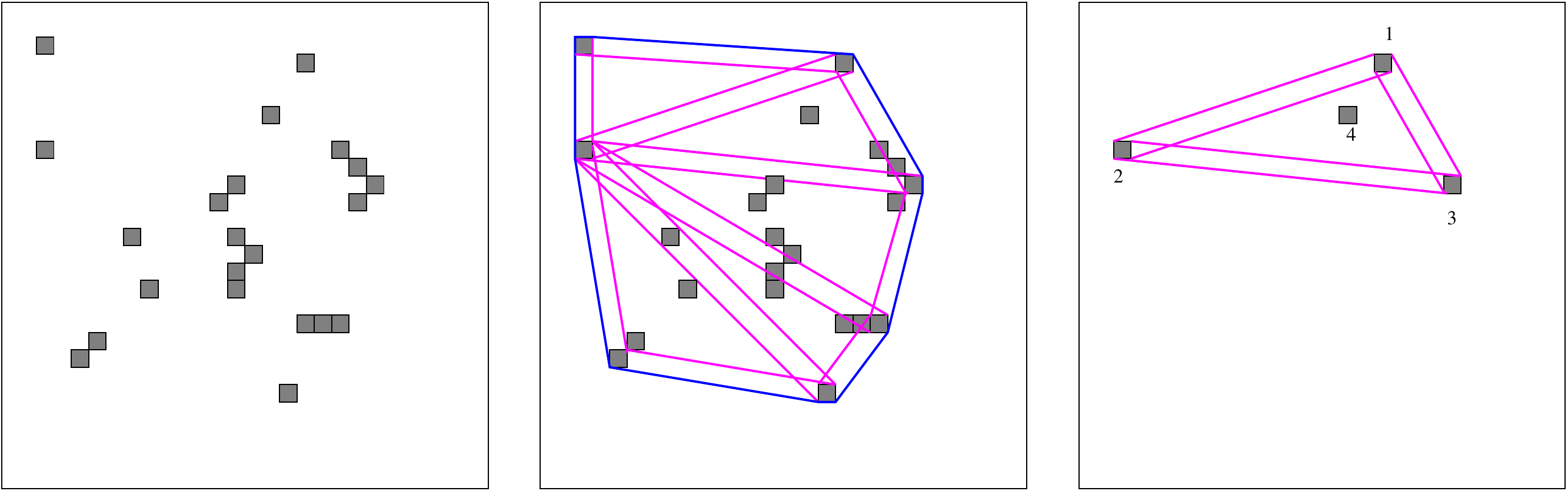}
\caption{Left: The set of rich cells in $Q$ (each rich cell is shaded).
  Center: the star triangulation~$K$ from a boundary cell in $\C$. Here $|\R|=22$ and $|\C|=7$.
  Segments in $\S$ are in bold lines.
  Right: a set of four rich cells as in Lemma~\ref{lem:4rich}.
}
\label{fig:grid}
\end{figure}

\begin{claim} \label{claim:intersect}
  The segments in $\S$ intersect at most $8 c' k^{5/3}$ cells in $\Sigma$.
\end{claim}

We verify the claim.
Observe that for each $i=1,\ldots,2|\C|-3$, the segments in $\S$ associated with $\tau_i$ intersect at most $4k$
cells in $\Sigma$ (recall that $\Sigma$ consists of $k^2$ cells).
Indeed, if the translation vector corresponding to the pair of cells in $\tau_i$ makes an angle of at most $45^\circ$
with the $x$-axis, $\tau_i$ can intersect at most four cells in each column. Otherwise $\tau_i$ can intersect at most
four cells in each row.  Since no point of $A$ lies on a cell boundary and $|\C| \leq c' k^{2/3}$, the claim follows.

\medskip
Since $P=\conv(\R)$, there are no rich cells in the exterior of $P$.
Moreover, every rich cell intersecting some $\tau_i$ intersects a segment in $\S$.
Note that for $k > (24 c' \cdot \alpha^2)^3$ we have
$ \frac{k^2}{3\alpha^2} - 8 c'\, k^{5/3} >0$.  To this end let
\begin{equation} \label{eq:alpha^6}
  k(\alpha) = \lceil (24 c' \cdot \alpha^2)^3 \rceil +1,
\end{equation}
and recall that we have set $k=k(\alpha)$ in the beginning of the proof. 
It follows that there exists at least one rich cell completely inside one of the triangles of the
star triangulation from $\sigma_0$. More precisely, if $\sigma_0,\sigma_j,\sigma_{j+1}$,
is such a triangle (triple) of rich cells and $\sigma'$ is a rich cell inside the triangle,
then
\begin{equation} \label{eq:empty}
\conv(\sigma_0 \cup \sigma_j) \cap \sigma' =\emptyset, \ \
\conv(\sigma_0 \cup \sigma_{j+1}) \cap \sigma' =\emptyset, \ \text{ and } \
\conv(\sigma_j \cup \sigma_{j+1}) \cap \sigma' =\emptyset.
\end{equation}

Setting $\sigma'_1:= \sigma_0$, $\sigma'_2:= \sigma_j$, $\sigma'_3:= \sigma_{j+1}$,
and $\sigma'_4:= \sigma'$, it is now easily verified that
for any four points $a_i \in \sigma'_i \cap A$, $i=1,2,3,4$,
we have $a_4 \in \Delta{a_1 a_2 a_3}$, as required.
\qed

\subparagraph{Final argument.}
We use the point set structure guaranteed by Lemma~\ref{lem:4rich}.
A very similar structure is highlighted and implicitly used in~\cite{PSS23}.
For completeness we include the proof tailored for our structure.

Recall that a cell $\sigma \in \Sigma$ is \emph{rich} if it contains at least $n/(3k^2)$ points in $A$.
Consider four rich cells $\sigma_1,\sigma_2,\sigma_3,\sigma_4$, such that
for any four points $a_i \in \sigma_i \cap A$, $i=1,2,3,4$, we have $a_4 \in \Delta{a_1 a_2 a_3}$.
Let $A_i = A \cap \sigma_i$, for $i=1,2,3,4$. Remove points from each of the four cells, if needed,
until there are exactly $\lceil n/(3k^2) \rceil$ points in $A$ in each of these cells.
Let us denote the resulting sets as $B_i \subseteq A_i$, for $i=1,2,3,4$, where
$|B_1|=|B_2|=|B_3|=|B_4| = m = \lceil n/(3k^2) \rceil$.

Recall that we are in the case $n \geq n_0$ and that
$m = \lceil n/(3k^2) \rceil$, where $k=k(\alpha)$ is a fixed integer.
Applying Lemma~\ref{lem:combined} with $P:=A$ and $Q:=B$
and Lemma~\ref{lem:special}
yields that the edge-set of the complete geometric graph $K_n[A]$ can be decomposed into at most
\begin{equation} \label{eq:combined}
  n-4m+ 3m = n-m \leq \left( 1 - \frac{1}{3k^2} \right) n
\end{equation}
plane subgraphs. Setting $c(\alpha) = 1 - \frac{1}{3k^2(\alpha)}$
and recalling~\eqref{eq:alpha^6} completes the proof of inequality~\eqref{eq:c(alpha)}.
\qed

\subparagraph{Note.}
Next, we show that $c(\alpha) \geq 1/2$ in Theorem~\ref{thm:partition},
\ie, some sets require at least $n/2$ plane subgraphs in the partition.
For simplicity, we give a grid example, where $\alpha \leq \sqrt2 + \eps$,
for a  small $\eps>0$. (A~suitable example can be found for every $\alpha > \alpha_0$.)
Let $n=k^2-1$, where $k=2a+1$. Consider the $n/2$ integer points
with positive $y$-coordinates or zero $y$-coordinate and negative $x$-coordinate in the
lattice section $\{-a,\ldots,a\}^2$, suitably perturbed to avoid collinearities.
Refer to Fig.~\ref{fig:dense}.
\begin{figure}[htbp]
\centering
\includegraphics[scale=1.2]{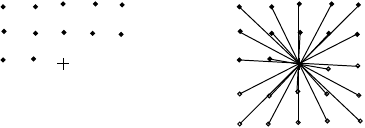}
\caption{A dense set of $24$ points with a crossing family of size $12$. The origin (marked with a cross)
  is not part of the set.}
\label{fig:dense}
\end{figure}
Add to these points the $n/2$ reflections with respect to the origin, suitably perturbed to avoid collinearities.
Observe that the resulting point set has $n$ points and admits a crossing family of size $n/2$
consisting of  $n/2$ edges connecting the $n/2$ initial points with their reflections. 
Consequently, any partition of the corresponding complete geometric graphs into plane subgraphs consists of at
least $n/2$ such subgraphs.

\section{Concluding remarks} \label{sec:remarks}

\textbf{A.} There are many geometric results for finite point sets that can be strengthened
under the assumption that the set is dense, for instance in the case of crossing families or the classic
Erd\H{o}s--Szekeres problem on points in convex position, as explained next; see also~\cite[Ch.~10]{BMP05}.
In $1935$, Erd\H{o}s and Szekeres~\cite{ES35} proved, as one of the first Ramsey-type results in combinatorial geometry,
that every set of $n$ points in general position in the plane contains $\Omega(\log{n})$ points in convex position,
and some $25$ years later showed~\cite{ES60} that this bound is tight up to the multiplicative constant.
According to the current best (asymptotic) upper bound, due to Suk~\cite{Suk17}, every set of $n$ points in general position
in the plane contains $(1-o(1))\log{n}$ points in convex position, and this bound is tight up to lower-order terms.

In contrast, a classic result of Valtr given below specifies a much larger threshold for the maximum size
of a subset in convex position in a density-restricted point set~\cite{Va92}:
  For every $\alpha \geq \alpha_0$ there exists $\beta=\beta(\alpha)>0$ such that any
  set of $n$ points in general position in the plane satisfying $D(A) \leq \alpha n^{1/2}$,
  contains a subset of $\beta n^{1/3}$ points in convex position.
 On the other hand, for every $n \in \NN$ there exists an $n$-element point set $A \subset \RR^2$ in general position,
 satisfying $D(A) = O(n^{1/2})$, in which every subset in convex position has at most $O(n^{1/3})$ points.
 In particular, a suitable small perturbation of the points in a piece of the integer lattice has this property.

 \smallskip
 \textbf{B.}  Pach and Solymosi~\cite{PS99} gave a concise characterization of point sets admitting a cross
ing family of size $n/2$. They showed that a set $P$ of $n$ points in general position in the plane ($n$ even) admits
a perfect matching with pairwise crossing segments if and only if $P$ has precisely $n$ halving lines.
A \emph{halving line} for such a set is a line incident to two points of the set and leaving exactly $n/2 -1$
points in each of the two open halfplanes it determines~\cite[Ch.~8.3]{BMP05}.

As defined by Dillencourt, Eppstein, and Hirschberg~\cite{DEH00}, the \emph{geometric thickness}
of an abstract graph $G$ is the minimum $k \in \NN$ such that $G$ has a drawing as a geometric graph
whose edges can be partitioned into $k$ plane subgraphs. The authors proved that the geometric thickness
of $K_n$ is between $\lceil (n/5.646) + 0.342 \rceil$ and $\lceil n/4 \rceil$.
As pointed out in~\cite{BHRW06}, the difference between Problem~\ref{prob:partition} and determining
the geometric thickness of $K_n$ is that Problem~\ref{prob:partition} deals with all possible drawings of $K_n$
whereas geometric thickness asks for the best drawing.

Decompositions of the edge-set of a complete geometric graph on $n$ points into the minimum number
of families of pairwise disjoint edges (resp., pairwise intersecting edges),
have been studied among others, by Araujo, Dumitrescu, Hurtado, Noy, and Urrutia~\cite{ADH+05}.

Recently, Obenaus and Orthaber~\cite{OO21} gave a negative answer to the question of whether every complete
geometric graph on $n$ vertices ($n$ even) can be partitioned into $n/2$ plane subgraphs. See also~\cite{AOO+22}.
As such, $n/2 +1$ is a lower bound in some instances on the number of such subgraphs in Problem~\ref{prob:partition}.

\smallskip
\textbf{C.}  If $X$ is a finite point set in the plane, every point of $\conv(X)$ can be expressed as a convex combination
of at most $3$ points in $X$. This implies that every point set in general position that is not in convex position contains
a subset of $4$ points that are not in convex position, \ie,  a four-tuple $a,b,c,d \in X$ such that $d \in \Delta{abc}$.

Our Theorem~\ref{thm:partition} gives the following quantitative version of Carath\'eodory's Theorem
for $\alpha$-dense sets.

\begin{corollary} \label{cor:4-tuples}
  Let $A$ be a set of $n$ points in the plane, with $D(A) \leq \alpha n^{1/2}$, for some $\alpha \geq \alpha_0$.
  Then there exist at least $c n^4$ four-tuples $a_1,a_2,a_3,a_4 \in A$ such that $a_4 \in \Delta{a_1 a_2 a_3}$,
  where $c=c(\alpha)>0$ is a constant.
\end{corollary}
\begin{proof}
  We may assume that $n$ is large enough.
Recall that $m = \lceil n/(3k^2) \rceil$, where $k=k(\alpha)$ is a fixed integer.
By Lemma~\ref{lem:4rich}, there exist four rich cells $\sigma'_1,\sigma'_2,\sigma'_3,\sigma'_4$,
such that for any four points $a_i \in \sigma'_i \cap A$, $i=1,2,3,4$, we have $a_4 \in \Delta{a_1 a_2 a_3}$.
Consequently, the number of (ordered) $4$-tuples with this property is at least
$ m^4 \geq n^4/(81 k^8)$. Setting $c(\alpha) = 3^{-4} k^{-8}(\alpha)$ completes the proof
of the lower bound. On the other hand, the total number of such $4$-tuples is clearly less than $n^4$.
\end{proof}

\smallskip
\textbf{D.} The proof of Corollary~\ref{cor:random} is straightforward.
  Subdivide $U=[0,1]^2$ into $25$ smaller axis-parallel squares as in Fig.~\ref{fig:random}.
  Consider the four subsquares: $\sigma_1=[0,1/5]^2$, $\sigma_2=[4/5,1] \times [0,1/5]$,
  $\sigma_3=[2/5,3/5] \times [4/5,1]$, and $\sigma_4=[2/5,3/5] \times [1/5,2/5]$.
\begin{figure}[htbp]
\centering
\includegraphics[scale=0.7]{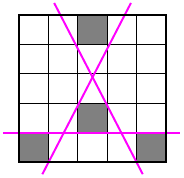}
\caption{The four distinguished subsquares are shaded.}
\label{fig:random}
\end{figure}

The expected number of points in each subsquare is $n/25$.
With probability tending to $1$ as $n \to \infty$, each of the four subsquares contains at least
$n/50$ points in $A$.

Observe that any line connecting a point in $\sigma_1$ with a point in $\sigma_3$ leaves $\sigma_4$ below.
By symmetry, any line connecting a point in $\sigma_2$ with a point in $\sigma_3$ leaves $\sigma_4$ below.
Third, any line connecting a point in $\sigma_1$ with a point in $\sigma_2$ leaves $\sigma_4$ above.
As such, a structure analogous to that in Lemma~\ref{lem:4rich} is obtained, and the corollary follows.
\qed

 \smallskip
 \textbf{E.}  Can the dependency of $c(\alpha)$ on $\alpha$ in~\eqref{eq:c(alpha)} be improved?
 Or completely eliminated?

\end{document}